\documentclass[12pt,a4paper]{amsart}
\usepackage[english]{babel}
\usepackage{lmodern}
\usepackage{newlfont}
\usepackage{booktabs}
\usepackage{color}
\usepackage[T1]{fontenc}
\usepackage[utf8]{inputenc}
\usepackage{indentfirst}
\usepackage{amsmath}
\usepackage{amsmath,amssymb}
\usepackage{amsthm}
\usepackage{geometry}
\usepackage{mathtools}
\usepackage{listings}
\usepackage{amsfonts}
\usepackage{braket}
\usepackage{emptypage}
\usepackage{newlfont}
\usepackage{amssymb}
\usepackage{graphicx}
\usepackage[italian]{varioref}
\usepackage{accents}
\usepackage{comment}
\usepackage{stmaryrd}
\usepackage{pgfplots}

\makeatletter
\labelformat{equation}{\tagform@{#1}}
\makeatother
\usepackage{hyperref}
\usepackage{relsize}
\usepackage{faktor}
\usepackage{enumerate}
\usepackage{fancyhdr}

\usepackage[colorinlistoftodos]{todonotes}
\usepackage{tikz}
\usetikzlibrary{arrows.meta, positioning}

\DeclarePairedDelimiter{\abs}{\lvert}{\rvert}

\DeclareMathOperator{\hess}{Hess}

\renewcommand{\phi}{\varphi}
\renewcommand{\epsilon}{\varepsilon}
\renewcommand{\vec}{\boldsymbol}

\newcommand{\numberset}{\mathbb}
\newcommand{\eps}{\varepsilon}

\newcommand{\bb}{\beta}
\newcommand{\N}{\numberset{N}}

\newcommand{\R}{\numberset{R}}

\newcommand{\norm}[1]{\left\|#1\right\|}

\newcommand{\tangvet}{\textsl{t}}
\newcommand{\orig}{\mathbf{0}}

\theoremstyle{definition}
\newtheorem{definition}{Definition}[section]
\theoremstyle{definition}                                                                         
\newtheorem{definizione}{Definizione}[section]
\theoremstyle{definition}                                                                         
\newtheorem{rmk}[definizione]{Remark}
\theoremstyle{plain}                                                                              
\newtheorem{thm}[definizione]{Theorem}
\theoremstyle{plain}    
\newtheorem{prop}[definizione]{Proposition}
\theoremstyle{plain}     
\newtheorem{lemma}[definizione]{Lemma}
\theoremstyle{plain}
\newtheorem{cor}[definizione]{Corollary}
\theoremstyle{definition}

\numberwithin{equation}{section}

\begin{document}
\title[Critical points of solutions of Robin problems]{On the critical points of solutions of Robin boundary problems}
\thanks{This work was supported by INdAM-GNAMPA.\\
F.D.R. has been supported by the Juan de la Cierva fellowship Grant JDC2022-048890-I funded by MICIU/AEI/10.13039/501100011033 and by the “European Union NextGenerationEU/PRTR”.
}

 \author{Fabio De Regibus}
\address{ \vspace{-0.4cm}
	\newline 
	\textbf{{\small Fabio De Regibus}} 
	\vspace{0.15cm}
	\newline \indent  
	Departamento de An\'alisis Matem\'atico, Universidad de Sevilla, C/ Tarfia S/N (Campus Reina Mercedes), 41012, Seville, Spain.}
\email{fderegibus@us.es}

 \author{Massimo Grossi}
\address{ \vspace{-0.4cm}
	\newline 
	\textbf{{\small Massimo Grossi}} 
	\vspace{0.15cm}
	\newline \indent  
	Dipartimento di Scienze di Base e Applicate per l'Ingegneria, Universit\`a di Roma ``La Sapienza'', Via Antonio Scarpa 14, 00161, Roma, Italy. }
\email{massimo.grossi@uniroma1.it}

\begin{abstract}
In this paper we prove the uniqueness of the critical point for \emph{stable} solutions of the Robin problem
\begin{equation*}
\begin{cases}
-\Delta u=f(u)&\text{in }\Omega\\
u>0&\text{in }\Omega\\
\partial_\nu u+\beta u=0&\text{on }\partial\Omega,
\end{cases}
\end{equation*}
where $\Omega\subseteq\R^2$ is a smooth and bounded domain with strictly positive curvature of the boundary, $f\ge0$ is a smooth function and $\beta>0$. Moreover, for $\bb$ large the result fails as soon as the domain is no more convex, even if it is very close to be: indeed, in this case it is possible to find solutions with an arbitrary large number of critical points.
\end{abstract}
\maketitle

\section{Introduction and main results}
\label{S1}

Let $\Omega$ be a smooth, bounded, and simply connected domain of $\R^2$. We are interested in the number of critical points of  solutions $u_\bb$ to the following semilinear elliptic problem with Robin boundary conditions
\[
\hypertarget{PB}{(\text{P}_\bb)\qquad}
\begin{cases}
-\Delta u=f(u)&\text{in }\Omega\\
u>0&\text{in }\Omega\\
\partial_\nu u+\bb u=0&\text{on }\partial\Omega,
\end{cases}
\]
where $f:\R\to\R$ is a smooth function, $\nu$ denotes the outer unit normal vector to $\partial\Omega$ and $\bb\in\R$ is a positive parameter.

This is a classical problem in partial differential equations, and it is well known that the shape of the solutions is strongly influenced by the geometry of the domain $\Omega$ and, of course, by the nonlinearity $f$.

A first interesting result linking the geometry and the topology of the domain with the geometry of the solution $u$ can be deduced from the Poincaré-Hopf Theorem, see for instance,~\cite{Mbook}. In particular, it follows that if $u$ is a positive solution of $-\Delta u=f(u)$ such that $\partial_\nu u\not=0$ on $\partial\Omega$, with isolated critical points $x_1,\dots,x_k$, then one can prove that
\[
\sum_{i=1}^k\mathrm{ind}(\nabla u,{x_i})=(-1)^N\chi(\Omega),
\]
where $\chi(\Omega)$ denotes the Euler characteristic of $\Omega$. Recall that if $y$ is an isolated zero of a vector field $T$, then we denote by $\mathrm{ind}(T,y):=\deg(T,B_r(y),\orig)$ the Brouwer degree of $T$ in a ball centered at $y$ with radius $r>0$ small enough (we point out that the previous formula is general and does not depend on the fact that $u$ is a solution of an elliptic equation). Furthermore, it is natural to ask when the above sum reduces to the minimal possible number of terms. Hence, if $\Omega$ is a contractible domain, we have $\chi(\Omega)=1$, and the question becomes to investigate under which conditions uniqueness of the critical point holds.

We start by recalling some results in the Dirichlet case. In this case, the literature is very broad, and it is impossible to provide a complete list of references; therefore, we mainly focus on those results that are more closely related to the present paper.

A first important result was proved by Makar-Limanov for the torsion problem, i.e.~$f\equiv1$. In~\cite{ML71}, it is shown that if $\Omega$ is a convex domain in $\R^2$, then the solution of the torsion problem has a unique nondegenerate critical point. Moreover, as mentioned in~\cite{KawBook}, the strict $1/2$-concavity of the solution is established, namely, $u^{1/2}$ is a concave function.

A similar result holds for the first Dirichlet eigenfunction, $f(u)=\lambda u$, as proved by Brascamp and Lieb~\cite{bl} (see also the paper by Acker, Payne, and Philippin~\cite{app}). In this case, the solution turns out to be $\log$-concave, that is, $\log(u)$ is concave. In particular, in both cases the superlevel sets $\{u>c\}$ are convex.

A very general result on the uniqueness of the critical point of solutions can be deduced from the seminal paper~\cite{gnn} by Gidas, Ni, and Nirenberg. Indeed, if $f$ is a Lipschitz continuous function and $\Omega\subseteq\R^N$ is a smooth and bounded domain that is symmetric with respect to the hyperplanes $x_i=0$ for any $i=1,\dots,N$ and convex in the directions $x_1,\dots,x_N$, then $u$ has exactly one critical point. Moreover, the superlevel sets are star-shaped with respect to the origin (but in general they are not convex, see the work by Hamel, Nadirashvili, and Sire~\cite{hns}).

Removing the symmetry assumptions on $\Omega$, while keeping $f$ very general, is a challenging task. A significant contribution in this direction is the result in~\cite{cc} by Cabré and Chanillo, where uniqueness of the critical point is proved for stable solutions in convex planar domains with strictly positive curvature of the boundary $\partial\Omega$.
We recall that a solution $u$ is said to be \emph{stable} if the linearized operator at $u$ is nonnegative definite. This theorem was extended to the case of vanishing curvature in~\cite{dgm}. We point out that it remains an open problem to determine whether the superlevel sets of stable solutions are convex, see again~\cite{hns}.
The preceding results have also been extended to the setting of manifolds, see the recent paper~\cite{GP23}.

All the results mentioned above hold in convex domains, and it is well known that, in general, uniqueness of the critical point cannot be expected in nonconvex domains, as shown, for instance, by the torsion problem in a dumbbell domain (see, e.g.,~\cite{sperb}). In some cases, however, it is still possible to recover uniqueness of the critical point in nonconvex domains under suitable assumptions, see, for example,~\cite{BDRG23} and~\cite{GG24}.

Even when $\Omega$ is "not far" from being convex and the minimum of the curvature of $\partial\Omega$ is close to zero, the situation may change drastically. Indeed, not only can uniqueness of the critical point fail, but it may even be impossible to get a bound on the number of critical points, see~\cite{GG23}.\\

In the case of Robin boundary conditions, to the best of our knowledge, there are relatively few results available. The first one we want to mention is due to Sakaguchi~\cite{Sak90}. Here, uniqueness of the critical point is proved for negative solutions of~\hyperlink{PB}{$(\text{P}_\bb)$} under the assumption that the nonlinearity $f$ satisfies $f'\le0$ and $f(0)<0$. 
As a particular case, the author also establishes the result for the solution of the torsion problem $f\equiv1$. Indeed, if $u$ satisfies $-\Delta u=1$ in $\Omega$ and $u_\nu+\bb u=0$ on $\partial\Omega$, then it is straightforward to observe that $v=-u$ satisfies $-\Delta v=-1$ in $\Omega$ together with the same boundary conditions as $u$. Therefore, Sakaguchi's result applies.

In this latter case, it is also possible to show that, if $\bb$ is sufficiently large (i.e.~if we are close to the Dirichlet case), $u_\bb$ is strictly $1/2$-concave, as proved by Crasta and Fragalà in~\cite{CF21}. They also show that the first Robin eigenfunction of the Laplacian is strictly $\log$-concave, again for large values of $\bb$. In particular, this implies\footnote{Note that, with the notation of~\cite{CF21}, $\log$-concavity means that $\log(u_\bb)$ has a positive definite Hessian matrix, which in particular implies uniqueness of the critical point of $u_\bb$.} uniqueness of the critical point for sufficiently large values of $\bb$.

It is worth pointing out that $\log$-concavity of the first Robin eigenfunction of the Laplacian no longer holds when $\bb$ is close to $0$, that is is if we are close to the Neumann case. This was proved by Andrews, Clutterbuck, and Hauer in~\cite{ACH21}.\\

In this paper, we study the number of critical points of stable solutions to problem~\hyperlink{PB}{$(\text{P}_\bb)$}, extending the result of Cabré and Chanillo in~\cite{cc} from the Dirichlet case to the Robin case. We also point out that, in general, there is no hope of covering the Neumann case, see Remark~\ref{noNeum}.

Finally, we show that as soon as the domain is no longer convex, even if it is very close to being so, uniqueness of the critical point is lost for large values of $\bb$, as in the Dirichlet case.\\

Before stating our main result, we recall the definition of a stable solution in the case of Robin boundary conditions.

\begin{definition}
\label{def:stab}
A solution $u_\bb$ of problem~\hyperlink{PB}{$(\text{P}_\bb)$} is said to be \emph{stable} if the linearized operator at $u_\bb$ is nonnegative definite, that is, if for all $\phi\in H^1(\Omega)$ one has
\[
\int_{\Omega}|\nabla\phi|^{2}\,dx+\bb\int_{\partial\Omega}\phi^2\,d\sigma-\int_{\Omega} f'(u_\bb)|\phi|^{2}\,dx\ge0.
\]
\end{definition}

We can now state the main result of the paper

\begin{thm}
\label{MainThm}
Let $\Omega\subseteq\R^2$ be a smooth and bounded domain such that the curvature $\kappa$ of its boundary is strictly positive, i.e.~$\kappa>0$ on $\partial\Omega$, and let $f\ge0$. Assume that there exists $\beta_0>0$ such that, for all $\bb>\bb_0$, problem~\hyperlink{PB}{$(\text{P}_\bb)$} admits a branch of stable solutions $u_\bb$ that is continuous in $\bb$ with respect to the $L^\infty(\Omega)$ norm and such that
\[
\norm{u_\bb- u_D}_{L^\infty(\Omega)}\to 0,\quad\text{as }\bb\to+\infty,
\]
where $u_D$ is a positive solution of the associated Dirichlet problem. i.e. 
\[
\hypertarget{PD}{(\text{P}_D)\qquad}
\begin{cases}
-\Delta u_D=f(u_D)&\text{in }\Omega\\
u_D=0&\text{on }\partial\Omega.
\end{cases}
\]Then, for all $\bb>\bb_0$, the solution $u_\bb$ has a unique critical point, which is a nondegenerate maximum.
\end{thm}

\begin{rmk}
\begin{enumerate}
\item We emphasize that the previous result is not obtained by perturbing the corresponding Dirichlet problem in the limit $\beta\to\infty$. On the contrary, it holds for every positive value of $\beta$ as soon as there exists the continuous branch of positive solutions.
\item Note that assuming the existence of such a continuous branch is not truly restrictive, since this is often the case, as will be clear from Corollary~\ref{MainThm2}.
\end{enumerate} 
\end{rmk}

Next, we state as a corollary some important cases in which the preceding theorem applies.
Indeed, under the assumptions listed below, it is possible to show the existence of a continuous branch of positive solutions for all $\bb>0$ and then Theorem~\ref{MainThm}  applies.
\begin{cor}
\label{MainThm2}
Let $\Omega\subseteq\R^2$ be a smooth and bounded domain such that the curvature $\kappa$ of its boundary is strictly positive,  i.e.~$\kappa>0$ on $\partial\Omega$. Assume that $f$ satisfies one of the following conditions:
\begin{enumerate}[(i)]
\item $f(t)=\lambda_\bb t$, where $\lambda_\bb:=\lambda_\bb(\Omega)$ is the first Robin eigenvalue in $\Omega$;
\item $f(t)\equiv1$;
\item Let  $g$ be a fixed smooth function satisfying $g(0)>0$ and $g'(t)>0$ for all $t\ge0$. Consider $f(t)=\lambda g(t)$, for $\lambda\in(0,\lambda_D^*)$, where $\lambda_D^*:=\lambda_D^*(\Omega)>0$ is a constant depending only on the domain $\Omega$.
\end{enumerate}
If either (i) or (ii) holds, then the solution $u_\bb$ of problem~\hyperlink{PB}{$(\text{P}_\bb)$} has a unique critical point, which is a nondegenerate maximum, for all $\bb>0$.  
If (iii) is satisfied, the same conclusion holds provided that $u_\bb$ is a minimal solution (see Theorem~\ref{ThmSTAB}).
\end{cor}

\begin{rmk}
\begin{enumerate}[(i)]
\item In the case of the first eigenfunction, it is interesting to point out that uniqueness of the critical point holds true also for small values of $\beta>0$, despite there exist convex domains where $u_\beta$ is not $\log$-concave for such small values of $\bb$, as shown by Andrews, Clutterbuck and Hauer in~\cite[Corollary 1.5]{ACH21}. A similar phenomenon occurs for the Dirichlet problem, see~\cite{hns}.
\item We recall that the Robin torsion function is stable and unique. However, this result is not new, since it already follows from~\cite[Theorem 2]{Sak90}.
\item Typical examples of nonlinearities $f$ satisfying (iii) are the model nonlinearities $f(t)=\lambda e^t$ and $f(t)=\lambda (1+t)^p$, with $p>0$ and $\lambda>0$. Let us also point out that $\lambda_D^*$ is such that no Dirichlet solutions exist if $\lambda>\lambda_D^*$. We refer to Section~\ref{s3} for further details.
\end{enumerate}
\end{rmk}

\begin{rmk}
As can be deduced from Corollary~\ref{MainThm2}, uniqueness of the critical point can be established for any value of $\bb>0$. Indeed, in all cases (i)-(iii) of Corollary~\ref{MainThm2}, problem~\hyperlink{PB}{$(\text{P}_\bb)$} admits a branch of stable solutions $u_\bb$ that is continuous in $\bb$ with respect to the $L^\infty(\Omega)$ norm, for all $\bb>0$.
\end{rmk}

\begin{rmk}
\label{noNeum}
We point out that even in cases where uniqueness of the critical point holds for every $\bb>0$, it is not possible to extend the result to the Neumann boundary condition
\[
\partial_\nu u=0,\quad\text{on }\partial\Omega,
\]
that is, to the case $\bb=0$ in~\hyperlink{PB}{$(\text{P}_\bb)$}. Indeed, it is a classical result due to Casten and Holland~\cite{CH78} and Matano~\cite{Mat79} that, in this case, nonconstant stable solutions do not exist in convex domains.
\end{rmk}

\begin{rmk}
Finally, we briefly comment on the assumptions in Theorem~\ref{MainThm}.
\begin{enumerate}[(i)]
\item The conclusion of Corollary~\ref{MainThm2}, and hence that of Theorem~\ref{MainThm}, is no longer true as soon as the domain is not convex, even if it is very close to being so. Indeed, in this case it is possible to construct solutions with an arbitrarily large number of critical points for $\bb$ large. We refer to Remarks~\ref{rmkS1} and~\ref{rmkS2} in Section~\ref{S3} for further details. It remains an open problem to treat the case of general convex domains in which the curvature may vanish at some points.
\item The stability assumption appears to be difficult to remove. Indeed, even for solutions of Morse index $1$ in the Dirichlet case, it is an open problem to determine whether uniqueness of the critical point holds.
\item In higher dimensions, the problem becomes even more involved. Indeed, in the Dirichlet case, uniqueness is known only in specific situations, such as the torsion problem or the first eigenfunction. Moreover, it is possible to construct smooth and bounded domains in $\R^N$, $N\ge3$, with strictly positive mean curvature - but not convex - such that the solution of the torsion problem on such domains admits a finite but arbitrarily large number of critical points, see~\cite{DRG}.
\item If we allow $f$, and consequently $u$, to change sign, then in general no control on the number of critical points can be expected. Indeed, examples of solutions with many critical points can be constructed even in the ball, see, for instance,~\cite{emp}. We also mention~\cite{DRG2}, where it is proved that the second Dirichlet eigenfunction has exactly two nondegenerate critical points in convex planar domains with large eccentricity.
\end{enumerate}
\end{rmk}

The proof of Theorem~\ref{MainThm} proceeds as follows. We first establish the result under the additional assumption that the zero sets of the partial derivatives of $u_\bb$ contain no singular points in $\overline\Omega$, in short $M_\theta^\bb=\emptyset$ for all $\theta\in[0,2\pi)$, see~\ref{def:M} for the precise definition of $M_\theta^\bb$. Note that this condition is satisfied when $\bb$ is sufficiently large (see Lemma~\ref{lemma:MthetaBD2}). Moreover, if $M_\theta^\bb=\emptyset$ for all $\theta\in[0,2\pi)$ and for all $\bb\ge\bb_1$, we can provide a precise description of the geometry of $u_\bb$ by establishing several intermediate results. In particular, we give a detailed description of the geometry of the nodal lines $N_\theta^\bb$ (see Lemma~\ref{lemma:lin:nod}). This precise understanding of the structure of $u_\bb$ allows us to prove that $\bb_1=\bb_0$ (Proposition~\ref{prop:beta}), thereby completing the proof.\\
Finally, to prove Corollary~\ref{MainThm2}, we show that  problem~\hyperlink{PB}{$(\text{P}_\bb)$} admits a branch of stable solutions $u_\bb$ that is continuous in $\bb$ with respect to the $L^\infty(\Omega)$ norm.

The paper is organized as follows. In the next section, after some preliminary results, we prove Theorem~\ref{MainThm} following the scheme outlined above. In Section~\ref{S3}, we address cases~(i)--(iii) of Corollary~\ref{MainThm2}.

\section{Proof of the main result}
\label{S2}

In this section, we prove Theorem~\ref{MainThm}. To this end, we first show that, even tough we assumed that the solutions $u_\bb$ converge, as $\bb\to+\infty$, to the solution of the corresponding Dirichlet problem in $L^\infty$, the convergence is indeed $\mathcal C^{2,\alpha}$.
We then establish our main result under the additional assumption on the nondegeneracy of the zero set of the partial derivatives of $u_\bb$. Finally, after deriving precise results on the geometry of $u_\bb$, we complete the proof of Theorem~\ref{MainThm}.

\subsection{Preliminary results}

We start by showing that the convergence to the solution of the corresponding Dirichlet problem is indeed $\mathcal C^{2,\alpha}$ and not only $L^\infty$.

\begin{lemma}
\label{ConvThm}
Under the assumptions of Theorem~\ref{MainThm}, as $\bb\to+\infty$, up to subsequences, one has
\[
u_\bb\to u_D,\quad\text{in }\mathcal C^{2,\alpha}(\overline\Omega),
\]
for some $\alpha\in(0,1)$, where $u_D$ is a positive and stable solution of the associated Dirichlet problem~\hyperlink{PD}{$(\text{P}_D)$}.
\end{lemma}

The proof of this result can be deduced from that of~\cite[Theorem 2.1]{CF21}, up to minor modifications. For the sake of completeness, we report it below.

\begin{proof}[Proof of Lemma~\ref{ConvThm}]
For all $\bb>\bb_0$, one has
\begin{align*}
\int_\Omega \abs{\nabla u_\bb}^2\,dx+\int_\Omega \abs{u_\bb}^2\,dx&\le\int_\Omega \abs{\nabla u_\bb}^2\,dx+\bb\int_{\partial\Omega} \abs{u_\bb}^2\,d\sigma+\int_\Omega \abs{u_\bb}^2\,dx\\
	&=\int_\Omega f(u_\bb)u_\bb\,dx+\int_\Omega \abs{u_\bb}^2\,dx\\
	&\le C_{\bb_0}\max_{[0,C_{\bb_0}]}\abs{f}\abs{\Omega}+C_{\bb_0}^2\abs{\Omega}.
\end{align*}
Hence, we can apply~\cite[Theorem 2.2]{CF21} and, in turn, the classical Sobolev embedding theorems to obtain
\begin{equation}
\label{eq:AA}
\norm{u_\bb}_{\mathcal C^{2,\alpha}(\overline\Omega)}\le C,
\end{equation}
for some $C:=C(\bb_0)>0$ and $\alpha\in(0,1)$. Eventually, up to possibly taking a smaller $\alpha$ and passing to a suitable subsequence, the Ascoli--Arzelà Theorem gives
\[
u_\bb\to u,\quad\text{in }\mathcal C^{2,\alpha}(\overline\Omega)\text{ as }\bb\to+\infty.
\]
It then follows trivially that $u$ satisfies the equation $-\Delta u=f(u)$ and that, for all $x\in\partial\Omega$,
\[
u(x)=\lim_{\bb\to+\infty}u_\bb(x)=\lim_{\bb\to+\infty}-\frac{\partial_\nu u_\bb(x)}\bb=0.
\]
To conclude the proof, it suffices to observe that $u>0$ in $\Omega$ by the Maximum Principle, and that for all  $\phi\in H^{1}_0(\Omega)$, which is the natural space to test Dirichlet stability, one has
\[
\int_{\Omega}|\nabla\phi|^{2}\,dx=\int_{\Omega}|\nabla\phi|^{2}\,dx+\bb\int_{\partial\Omega}\phi^2\,d\sigma\ge\int_{\Omega} f'(u_\bb)|\phi|^{2}\,dx,\quad\forall\bb>\bb_0.
\]
Then the stability of $u$ follows by passing to the limit in the preceding inequality.
\end{proof}

Estimate~\ref{eq:AA} in the proof of Lemma~\ref{ConvThm} holds under a slightly weaker assumption and will be useful in the sequel, so we state it here as a separate lemma.

\begin{lemma}
\label{prop}
Assume that there exists $\beta_0>0$ such that, for all $\bb>\bb_0$, problem~\hyperlink{PB}{$(\text{P}_\bb)$} has a stable solution $u_\bb$ that satisfies
\[
\norm{u_\bb}_{L^\infty(\Omega)}\le \widetilde C,
\]
for some $\widetilde C:=\widetilde C(\bb_0)>0$. In particular, this is the case of the assumptions of Theorem~\ref{MainThm}. Then
\[
\norm{u_\bb}_{\mathcal C^{2,\alpha}(\Omega)}\le C,\quad\text{for all }\bb>\bb_0,
\]
for some $\alpha\in(0,1)$ and some $C:=C(\bb_0)>0$.
\end{lemma}

The next lemma is nothing else than maximum principle and Hopf's boundary point lemma. The proof can be found, for instance, in~\cite{Sak90}, but we report it here for completeness.

\begin{lemma}
\label{LemmaMP}
Let $u\in\mathcal C^2(\overline\Omega)$ be a solution of of~\hyperlink{PB}{$(\text{P}_\bb)$} for some smooth function $f:\R\to[0,+\infty)$, and $\beta>0$. Then $u>0$ in $\overline\Omega$ and $\partial_\nu u<0$ on $\partial\Omega$.
\end{lemma}

\begin{proof}
Let $x\in\overline\Omega$ be such that 
\[
u(x)=\min_{\overline\Omega} u.
\]
By the Maximum Principle, we have $x\in\partial\Omega$ and, in particular, $\partial_\nu u(x)<0$. Hence, if we assume by contradiction that $u(x)=0$, we obtain
\[
\partial_\nu u(x)+\bb u(x)<0,
\]
which contradicts the Robin boundary condition. Therefore, $u>0$ in $\overline\Omega$ and
\[
\partial_\nu u=-\bb u<0,\quad\text{on }\partial\Omega.\qedhere
\]
\end{proof}

\begin{lemma}
\label{LemmaDeg}
Under the same assumptions as in Lemma~\ref{LemmaMP}, if the critical points of $u$ are isolated, then
\[
\deg(\Omega,\nabla u,\orig)=1,
\]
where $\deg$ denotes the topological degree of a vector field.
\end{lemma}

\begin{proof}
Since $\nabla u_\bb\cdot\nu=\partial_\nu u_\bb<0$ on $\partial\Omega$, we can apply the Poincaré--Hopf Theorem - see, for instance,~\cite{Mbook} \footnote{ In our case, the vector field $\nabla u$ points inward through the boundary, whereas in~\cite{Mbook} it is assumed to point outward. However, this makes no difference, since we are in dimension $N=2$.} - to the vector field $\nabla u_\bb$. The claim then follows from the fact that $\Omega$ is simply connected.
\end{proof}

From now until the end of Section~\ref{S2}, we assume that the hypotheses of Theorem~\ref{MainThm} are satisfied. In particular, $\Omega$ is a smooth and bounded domain such that the curvature $\kappa$ of its boundary is strictly positive, i.e.~$\kappa>0$ on $\partial\Omega$, and there exists $\bb_0>0$ such that, for all $\bb>\bb_0$, problem~\hyperlink{PB}{$(\text{P}_\bb)$} admits a branch of stable solutions $u_\bb$ that is continuous in $\bb$ with respect to the $L^\infty(\Omega)$ norm and such that
\[
\norm{u_\bb- u_D}_{L^\infty(\Omega)}\to 0,\quad\text{as }\bb\to+\infty,
\]
where $u_D>0$.

As in~\cite{cc}, we introduce the following notation. For every $\theta\in[0,2\pi)$, let $\vec{e_{\theta}}=(\cos\theta,\sin\theta)\in\mathbb S^1$, and set
\begin{equation}
\begin{split}
\label{def:M}
\partial_\theta u_\bb&:=\frac{\partial u_\bb}{\partial \vec{e_{\theta}}}=\cos\theta\, \partial_{x_1}u_\bb+\sin\theta\, \partial_{x_2}u_\bb,\\
N_{\theta}^\bb&:=\overline{\set{x\in{\Omega}\mid \partial_\theta u_\bb(x)=0}},\\
M_{\theta}^\bb&:=\set{x\in N_{\theta}^\bb \mid \nabla (\partial_\theta u_\bb(x))=\orig},
\end{split}
\end{equation}
where $u_\bb$ is always a stable solution of~\hyperlink{PB}{$(\text{P}_\bb)$}.

\begin{rmk}
\label{rmk:stab:Dir}
As in the Dirichlet case (see~\cite{cc}), for every $\theta\in[0,2\pi)$ there exists no nonempty domain $H\subseteq\Omega$, $|\Omega\setminus H|>0$, such that $\partial H\subseteq N_{\theta}^{\bb}$ (where the boundary of $H$ is understood as a subset of $\R^{2}$).

To prove this, let us denote by $\lambda_D(-\Delta+q(x),\mathcal U)$ the first eigenvalue of the operator $-\Delta+q(x)$ in a domain $\mathcal U$, with Dirichlet boundary conditions. If this were not the case, the function $\phi:=\partial_\theta u_\bb\chi_H$ would satisfy
\[
\int_\Omega\abs{\nabla\phi}^2\,dx-\int_\Omega f'(u_\bb)\phi^2\,dx=0,
\]
and therefore, by the monotonicity of the first Dirichlet eigenvalue with respect to domain inclusion
\begin{align*}
0&\ge\lambda_D(-\Delta-f'(u_\bb),H)\\
	&>\lambda_D(-\Delta-f'(u_\bb),\Omega)\\
	&\ge\inf_{\substack{\psi\in H^1(\Omega)\\ \norm{\psi}_{L^2(\Omega)}=1}}\int_\Omega\abs{\nabla\psi}^2\,dx+\beta\int_{\partial\Omega}\psi^2\,d\sigma-\int_\Omega f'(u_\bb)\psi^2\,dx\ge0,
\end{align*}
which is a contradiction. Here, the last inequality follows from the stability of $u_\bb$, see Definition~\ref{def:stab}.
\end{rmk}

\begin{lemma}
\label{lemma:tg}
Let $x\in M_\theta^\bb\cap\partial\Omega$ for some $\theta\in[0,2\pi)$. Then the tangent vector to $\partial\Omega$ in $x$ is parallel to $\theta$.
\end{lemma}
\begin{proof} 
Since the proof does not depend on $\bb$, we simply write $u$ instead of $u_\bb$. Without loss of generality, we may assume that $x=\orig$ and $\theta=(1,0)$. Since $x=\orig\in N_{\theta}^\bb$, it follows that $u_{x_1}(\orig)=0$. Lemma~\ref{LemmaMP} implies that $\partial_\nu u(\orig)<0$, and therefore the second component of $\nu(\orig)$ must be nonzero. In particular, up to a rotation, we may assume that $\nu(\orig)\cdot(0,1)<0$. Finally, the condition $\orig\in M_{\theta}^\bb$ yields $u_{x_1x_1}(\orig)=u_{x_1x_2}(\orig)=0$. 

Since $\partial\Omega$ is of class $\mathcal C^2$, there exists a local parametrization $\phi\in\mathcal C^2(-2\eps,2\eps,\R)$ of $\partial\Omega$ in a neighborhood of $\orig$ such that
\begin{equation}
\label{par:bdd}
\partial\Omega\cap B_\eps(\orig)=\{(\tau,\phi(\tau)):-\eps<\tau<\eps\}.
\end{equation}
We point out that the convexity of $\Omega$ implies $\phi''(0)\ge0$\footnote{Since $\kappa>0$ on $\partial\Omega$, one actually has $\phi''(0)>0$, but the strict inequality is not needed here.}.
Then for all $\tau\in(-\eps,\eps)$, we can write
\[
\tangvet(\tau,\phi(\tau))=\frac{(1,\phi'(\tau))}{A(\tau)},\quad\text{and}\quad\nu(\tau,\phi(\tau))=\frac{(\phi'(\tau),-1)}{A(\tau)},
\]
where
\begin{equation}
\label{A(0)}
A(\tau):=(1+(\phi'(\tau))^2)^{1/2}\ge1.
\end{equation}
We recall that here, $\tangvet(x)$ and $\nu(x)$ denote, respectively, the unit tangent and the unit normal vectors to $\partial\Omega$ at a point $x\in\partial\Omega$.

Then it is enough to prove that $\phi'(0)=0$. Indeed, our claim is that the tangent vector $\tangvet(\orig)$ to $\partial\Omega$ is parallel to $\theta=(1,0)$, which is equivalent to $\phi'(0)=0$. Assume by contradiction that $\phi'(0)\neq0$. Multiplying the Robin boundary condition by $A(\tau)$, we can rewrite it as
\begin{equation}
\label{eq:bdd:par}
\begin{split}
0&=A(\tau)\partial_\nu u(\tau,\phi(\tau))+\beta A(\tau)u(\tau,\phi(\tau))\\
	&=u_{x_1}(\tau,\phi(\tau))\phi'(\tau)-u_{x_2}(\tau,\phi(\tau))+\beta A(\tau)u(\tau,\phi(\tau)).
\end{split}
\end{equation}
Differentiating the above identity with respect to $\tau$, evaluating it at $\tau=0$, and recalling that $u_{x_1}(\orig)=u_{x_1x_1}(\orig)=u_{x_1x_2}(\orig)=0$, we obtain
\[
-u_{x_2x_2}(\orig)\phi'(0)+\bb A'(0)u(\orig)+\bb A(0)u_{x_2}(\orig)\phi'(0)=0.
\]
Now, taking into account that $A'(0)=(A(0))^{-1}\phi'(0)\phi''(0)$ and $u_{x_2x_2}(\orig)=-f(u(\orig))<0$, together with the assumption $\phi'(0)\not=0$, we deduce
\[
\bb u_{x_2}(\orig)(A(0))^2+f(u(\orig))A(0)+\bb u(\orig)\phi''(0)=0.
\]
Finally, since $(A(0))^{-1}u_{x_2}(\orig)=-\partial_\nu u(\orig)>0$ and $\phi''(0)\ge0$, we conclude that $A(0)\le0$, which contradicts~\ref{A(0)}.
\end{proof}

\begin{lemma}
\label{lemma:MthetaBD2}
There exists $\bb_\infty>\bb_0$, such that for all $\bb>\bb_\infty$, one has
\[
M_\theta^\bb=\emptyset,
\]
for all $\theta\in[0,2\pi)$.
\end{lemma}
\begin{proof}
By contradiction, there exists a sequence $(\bb_n)_{n\in\N}\subseteq(\bb_0,+\infty)$, with $\bb_n\to+\infty$ as $n\to+\infty$, such that there exists $x_n\in M_{\theta_n}^{\bb_n}\cap\partial\Omega$, for some $\theta_n\in[0,2\pi)$, for all $n\in\N$.

Then, taking into account Lemma~\ref{ConvThm}, we have $u_{\bb_n}\to u_D$ in $\mathcal C^{2,\alpha}(\overline\Omega)$, for some $\alpha\in(0,1)$, as $n\to+\infty$, up to subsequences. In particular, there exists $ x\in\overline\Omega$, $\theta\in[0,2\pi)$ such that $ x\in M_{\theta}^{D}$, but this is in contradiction with the results in~\cite{cc}, see also~\cite[Proposition 1]{dgm}.
\end{proof}

\subsection{The case $M_\theta^\bb=\emptyset$}
\label{subsec1}

As explained in the introduction, in this subsection we prove several results concerning the geometry of $u_\bb$ (such as the uniqueness of the critical point and a precise description of the geometry of the nodal lines $N_\theta^\bb$), under the additional assumption that $M_\theta^\bb=\emptyset$ for all $\theta\in[0,2\pi)$ and for all $\bb\ge\bb_1$, that is, that the zero sets of the partial derivatives of $u_\bb$ do not contain singular points in $\overline\Omega$. Note that such a $\bb_1$ exists thanks to Lemma~\ref{lemma:MthetaBD2}.
In the next subsection, we will show that $\bb_1=\bb_0$, which will conclude the proof of our main result.\\

Thanks to this additional assumption, we can immediately deduce that all the critical points of $u_{\bb}$ are nondegenerate and isolated for $\bb\ge\bb_1$.

\begin{lemma}
\label{CorCP}
Assume that $M_\theta^\bb=\emptyset$, for all $\theta\in[0,2\pi)$ and for all $\bb\ge\bb_1$. Then for all $\bb\ge\bb_1$, given $x\in\Omega$ such that $\nabla u_{\bb}(x)=\orig$, one has
\[
\det\hess_{u_{\bb}}(x)\not=0.
\]
Moreover, all the critical points are isolated.
\end{lemma}

In the next lemma, we show that if $u_\bb$ has more than one critical point, then it necessarily has at least one saddle point, and vice versa.

\begin{lemma}
\label{LemmaIFF}
Assume that $M_\theta^\bb=\emptyset$ for all $\theta\in[0,2\pi)$ and for all $\bb\ge\bb_1$. Then, for all $\bb\ge\bb_1$, the solution $u_{\bb}$ has more than one critical point if and only if it has at least one nondegenerate saddle point.
\end{lemma}
\begin{proof}
Since $u_\bb$ has no minima inside $\Omega$ by the Maximum Principle, the claim follows by combining Lemma~\ref{LemmaDeg} and Lemma~\ref{CorCP}.
\end{proof}

We can now prove that $u_\bb$ has a unique critical point for all $\bb\ge\bb_1$, that is, we can establish our main result under the additional assumption that the zero sets of the partial derivatives of $u_\bb$ do not contain singular points for all $\bb\ge\bb_1$.
We prove the result under the slightly weaker assumption $M_\theta^\bb\cap\Omega=\emptyset$ for all $\theta\in[0,2\pi)$ and for all $\bb\ge\bb_1$, since this condition will be needed in the following.

\begin{lemma}
\label{thmBetaGrande}
Assume that $M_\theta^\bb\cap\Omega=\emptyset$ for all $\theta\in[0,2\pi)$ and for all $\bb\ge\bb_1$. Then, for all $\bb\ge\bb_1$, $u_{\bb}$ has a unique critical point, which is a nondegenerate maximum.
\end{lemma}
\begin{proof}
By contradiction, we assume that there exists $\bb\ge\bb_1$ such that $u_{\bb}$ has at least two critical points. Thus
\[
I:=\set{\bb\ge\bb_1:u_\bb\text{ has at least } 2 \text{ critical points}}\not=\emptyset,
\]
and we set
\[
\bar\bb:=\sup I\ge\bb_1.
\]
Note that $\bar\bb$ could be either finite or infinite, so we distinguish between the two cases. In the rest of the proof, all convergences are to be understood up to a subsequence.\\

\noindent\emph{Case $\bar \beta<+\infty$.}\\
Take $(\bb_n)_n\subseteq I$ such that $\bb_n\to\bar\bb$ as $n\to+\infty$.
Thanks to Lemma~\ref{LemmaIFF}, for all $n\in\N$ we can find $x_n\in\Omega$ such that
\[
\nabla u_{\bb_n}(x_n)=\orig,\quad\text{and}\quad\det\hess_{u_{\bb_n}}(x_n)<0.
\]
Letting $n$ go to $+\infty$, we clearly have $x_n\to x_{\bar\bb}\in\overline\Omega$, and moreover Lemma~\ref{prop}, together with the Ascoli-Arzelà Theorem, implies that
\[
u_{\bb_n}\to u_{\bar\bb},\quad\text{in }\mathcal C^{2,\alpha}(\overline\Omega)\text{ as }n\to+\infty,
\]
for some $\alpha\in(0,1)$. Then, since the critical points are nondegenerate by Lemma~\ref{LemmaDeg}, we have
\[
\nabla u_{\bar\bb}(x_{\bar\bb})=\orig,\quad\text{and}\quad\det\hess_{u_{\bar\bb}}(x_{\bar\bb})<0.
\]
Moreover, $x_{\bar\bb}\in\Omega$ by Lemma~\ref{LemmaMP}. This implies that we can find $r>0$ such that
\[
\deg(B_r(x_{\bar\bb}),\nabla u_{\bar\bb},\orig)=-1.
\]

Consider now another sequence $(\bb_k)_k\subseteq\R$ with $\bb_k>\bar\bb$ and $\bb_k\to\bar\bb$ as $k\to+\infty$ (note that the existence of such a sequence is guaranteed by the assumption on the branch of solutions). Again, Lemma~\ref{prop}, together with the Ascoli-Arzelà Theorem and the properties of the topological degree, implies that
\[
\deg(B_r(x_{\bar\beta}),\nabla u_{\bb_k},\orig)=-1,
\]
for all $k$ sufficiently large, where possibly $r$ is smaller than before. This means that $u_{\bb_k}$ has at least one nondegenerate saddle point and, in turn, at least two critical points. Then $\bb_k\in I$, which is clearly impossible since $\bb_k>\sup I$.\\

\noindent\emph{Case $\bar \beta=+\infty$.}\\
Take $(\bb_n)_n\subseteq I$ such that $\bb_n\to+\infty$ as $n\to+\infty$. Thanks to Lemma~\ref{ConvThm} and arguing as in the previous case, we can find $x\in\overline\Omega$ such that
\[
\nabla u_{D}(x)=\orig,\quad\text{and}\quad\det\hess_{u_{D}}(x)\le0,
\]
where $u_D$ is a positive and stable solution of the corresponding Dirichlet problem~\hyperlink{PD}{$(\text{P}_D)$}.
However, this is not possible since we know that $u_D$ has exactly one nondegenerate critical point in $\overline\Omega$, and it is a maximum, see~\cite[Theorem 1]{cc}.
\end{proof}

We can now give an accurate description of the nodal lines $N_\theta^\bb$ for all $\bb>\bb_1$. Moreover, we can show that $\partial_\theta u_\bb$ changes sign exactly twice on $\partial\Omega$.
This result will be a crucial step in the final part of the proof of our main theorem.

As a first step we show that $\partial_\theta u_\bb$ vanishes on isolated (in the boundary) points on $\partial\Omega$.

\begin{lemma}
\label{lemma:is:boundary} 
Given $\bb>\bb_0$, assume that $M_\theta^\bb=\emptyset$ for all $\theta\in[0,2\pi)$. Then, for all $\theta\in[0,2\pi)$, there are no subdomains of $\partial\Omega$ with nonempty interior (as subsets of $\partial\Omega$) on which $\partial_\theta u_\bb\equiv0$.
\end{lemma}

\begin{proof}
By contradiction, assume that there exists a connected set $\Gamma\subseteq\partial\Omega$ with nonempty interior such that $\partial_\theta u_\bb\equiv 0$ on $\Gamma$ for some $\theta\in[0,2\pi)$. From now on and until the end of the proof, we simply write $u$ instead of $u_\bb$. Hence, given $x\in\Gamma\setminus\partial\Gamma$, one has
\[
\partial_{\tangvet} u_\bb(x)=\partial^2_{\tangvet\tangvet} u_\bb(x)=0.
\]
On the other hand, taking into account that $M_\theta^\bb=\emptyset$, one has $u_\theta\neq 0$ (say, positive) in $B_\eps(x)\cap\Omega$ if $\eps$ is small enough. This means that the curvature of the boundary of the superlevel set $\{u_\theta>0\}$ at $x$ is given by
\[
-\frac{(\partial_{\tangvet} u_\bb(x))^2\partial^2_{\nu\nu} u_\bb(x)-2\partial_{\tangvet} u_\bb(x)\partial_{\nu} u_\bb(x)\partial^2_{\tangvet\nu} u_\bb(x)+(\partial_{\nu} u_\bb(x))^2\partial^2_{\tangvet\tangvet} u_\bb(x)}{|\nabla u_\theta|^2}=0,
\]
which contradicts the fact that it should coincide with the curvature of the boundary of $\Omega$, which is strictly positive.
\end{proof}

\begin{lemma}
\label{lemma:lin:nod}
Assume that $M_\theta^\bb=\emptyset$ for all $\theta\in[0,2\pi)$ and for all $\bb\ge\bb_1$. Then, for every $\bb\ge\bb_1$ and every $\theta\in[0,2\pi)$, $N_\theta^{\bb}$ is a smooth curve in $\Omega$, without self-intersections, such that 
\[
N_\theta^{\bb}\cap\partial\Omega=\{P_\theta^{\bb}, Q_\theta^{\bb}\},
\]
and $\partial_\theta u^{\bb}_{\,|\partial\Omega}$ changes sign at $P_\theta^{\bb}$ and $Q_\theta^{\bb}$.
\end{lemma}

\begin{proof}
By hypothesis, we have $M_\theta^{\bb}=\emptyset$ for all $\theta\in[0,2\pi)$ and for all $\bb\ge\bb_1$. Then, by the Implicit Function Theorem, $N_\theta^{\bb}$ is locally a smooth curve in $\Omega$. Thus, $N_\theta^{\bb}$ is a finite union of disjoint simple curves
\[
N_\theta^{\bb}=\bigcup_{i=1}^k\gamma_i,
\]
where each $\gamma_i=\gamma_i^{(\theta)}$, $i=1,\dots,k=k(\theta)\in\N$, is a closed curve or ends on $\partial\Omega$.
Recall that, by the preceding lemma, $\gamma_i\cap\partial\Omega$ has empty interior (as a subset of $\partial\Omega$). Moreover, by stability (see Remark~\ref{rmk:stab:Dir}), there is no nonempty domain $H\subseteq\Omega$, $|\Omega\setminus H|>0$, such that $\partial H\subseteq \gamma_i$, so each $\gamma_i$ terminates at two distinct points on $\partial\Omega$ (which are not endpoints of any other $\gamma_j$). It should be observed that there may exist points of $\partial\Omega$ lying in the interior of some $\gamma_i$; in such cases, $\gamma_i$ is tangent to $\partial\Omega$ from the inside at those points (note that here $\partial_\theta u_{\bb|\partial\Omega}$ does not change sign).
Finally, $\partial_\theta u_{\bb|\partial\Omega}$ changes sign at the two endpoints on $\partial\Omega$, otherwise the point would be an internal tangency and hence not an endpoint.

We now show that there exists $\theta\in[0,2\pi)$ such that $k(\theta)\ge2$.
Arguing by contradiction, assume that $k(\theta)=1$, i.e.~$N_{\theta}^\bb$ is connected, for all $\theta$. Since the conclusion of the lemma fails, there exists $x\in\gamma_1\cap\partial\Omega$ such that $\gamma_1$ is tangent to $\partial\Omega$ from the inside at $x$, for some $\theta'$. In this case, we can write $\gamma_1=\Gamma_1\cup\Gamma_2$ with $\Gamma_1\cap\Gamma_2=\{x\}$.

\begin{figure}[ht]

\begin{tikzpicture}[scale=1.2]

\coordinate (A) at  (0,0); \node at (A) {$\bullet$}; \node[below] at (A) {$x$};
\coordinate (B) at  (3,2.3);
\coordinate (C) at  (1,4);
\coordinate (D) at  (-3,1.5);

\draw [ultra thick] (A) to [out=0,in=270] (B);
\draw [ultra thick] (B) to [out=90,in=0] (C);
\draw [ultra thick] (C) to [out=180,in=90] (D);
\draw [ultra thick] (D) to [out=270,in=180] (A);

 \node[above] at (-1.2,3.8) {$\Omega$};

\draw [ultra thick, blue] (A) to [out=0,in=250] (C);	\node[blue] at (1.2,1) {$\Gamma_2$};
\draw [ultra thick, blue] (A) to [out=180,in=35] (D);	\node[blue] at (-2.2,1.8) {$\Gamma_1$};

\draw [ultra thick, red] (-2.2,0.8) to [out=85,in=200] (2.4,2.7);
\draw[ultra thick, red,dash pattern={on 2pt off 1pt}] (-2.2,0.8) -- (-2.2,0.5);
\draw[ultra thick, red,dash pattern={on 2pt off 1pt}] (2.4,2.7) -- (2.7,2.8);
\node at (-2.17,1) {$\bullet$}; \node[left] at (-2.15,1) {$y$};
\node at (0.85,2.37) {$\bullet$}; \node[below] at (1.2,2.37) {$x_M$};

\node[red] at (-0.5,2.6) {$N_{\tilde\theta}^\bb$};

\end{tikzpicture}
\caption{\label{lin_nod}The nodal line $\gamma_1=\Gamma_1\cup\Gamma_2$ and $N_{\tilde\theta}^\bb$.}
\end{figure}

Recall that, by Lemma~\ref{thmBetaGrande}, $u_\bb$ has a unique critical point, which is a nondegenerate maximum denoted by $x_M\in\Omega$.
Thus, the unique critical point $x_M$ lies on one of the two arcs $\Gamma_i$, say $x_M\in\Gamma_2$. Note that $\Gamma_1$ separates $\Omega$ into two distinct connected components. Then, there exists a point $y$ in the component not containing $x_M$ such that $y\in N_{\tilde\theta}^\bb$ for some $\tilde\theta\neq\theta'$. Clearly, $x_M\in N_{\tilde\theta}^\bb$, and since $N_{\tilde\theta}^\bb$ is connected, it must also intersect $\Gamma_1$, which is impossible because the only critical point is $x_M$ (see Figure~\ref{lin_nod})

By the preceding discussion, it is immediate to see that $k\ge2$ implies 
\begin{equation}
\label{4pt}
N_{\tilde\theta}^\bb\cap\partial\Omega\supseteq\!\{z_1^\bb,\dots,z_4^\bb\}\text{ and }\partial_{\tilde\theta}u_{\bb|\partial\Omega}\text{ changes sign at }z_j^\bb,\, j=1,\dots,4,
\end{equation}
for some $\tilde\theta\in[0,2\pi)$.\\

Remember that if the conclusion of the lemma does not hold, then there exist $\bb>\bb_1$ and $\theta\in[0,2\pi)$ such that~\ref{4pt} is satisfied. Thus, we can consider the nonempty set
\[
B=\{\bb>\bb_1:\text{~\ref{4pt} holds}\}\not=\emptyset,
\]
and let $\bar\bb=\sup B$, which could be either finite or not.
Clearly,~\ref{4pt} implies that the set $B$ is open, thanks to the $\mathcal C^{2,\alpha}(\Omega)$ continuity - up to subsequences - with respect to $\bb$ of the solution $u_\bb$, by Lemma~\ref{ConvThm} or Lemma~\ref{prop} and the Ascoli-Arzelà Theorem. Then, at $\bb=\bar\bb$, the conclusion of the lemma holds, but there exists a sequence $(\bb_n)_{n\in\N}\subseteq B$ converging to $\bar\bb$. Since~\ref{4pt} is satisfied at $\bb_n$ for all $n\in\N$, say for $\theta=\theta_n$, while $\partial_\theta u_{\bar\bb\,|\partial\Omega}$ changes sign exactly twice for all $\theta$, there are two possible cases. Denoting by $\bar\theta$ the limit, up to subsequences, of $\theta_n$ as $n\to+\infty$, in the first case two of the points $z_j^{\bb_n}$ collapse to some $x\in\partial\Omega$, and $\partial_{\bar\theta}u_{\bar\bb\,|\partial\Omega}$ does not change sign at $x$; in the second case, three (or more) of the points $z_j^{\bb_n}$ collapse to some $x\in\partial\Omega$, and $\partial_{\bar\theta}u_{\bar\bb|\partial\Omega}$ changes sign at $x$, see Figure~\ref{fig4}.\\

\begin{figure}
\begin{tikzpicture}[>=Stealth, remember picture]

\node (A1)  {
\begin{tikzpicture}[scale=0.8]

\coordinate (A) at  (0,0); 
\coordinate (B) at  (3,2.3);
\coordinate (C) at  (1,4);
\coordinate (D) at  (-3,1.5);
\coordinate (E) at  (0.74,2.37);
\coordinate (F) at  (-0.5,3.85);
\coordinate (G) at  (2.53,1);
\coordinate (H) at  (-2.9,1);
\coordinate (I) at  (-2.7,0.7);

\draw [ultra thick] (A) to [out=0,in=270] (B);
\draw [ultra thick] (B) to [out=90,in=0] (C);
\draw [ultra thick] (C) to [out=180,in=90] (D);
\draw [ultra thick] (D) to [out=270,in=180] (A);

 \node[above] at (-1.2,3.8) {$\Omega$};

\draw [ultra thick, blue] (F) to [out=265,in=220] (E);
\draw [ultra thick, blue] (E) to [out=40,in=165] (G);
\node at (E) {$\bullet$}; \node[below] at (0.9,2.37) {$x_M$};

\node[blue] at (0.3,3) {$\gamma_2^{\theta_n}$};

\draw [ultra thick, blue] (D) to [out=0,in=0] (I);
\node at (H) {$\bullet$}; \node[left] at (H) {$x$};

\node[blue] at (-2.1,1.5) {$\gamma_1^{\theta_n}$};

\end{tikzpicture}
};

\node (A2) [right=1cm of A1] {
\begin{tikzpicture}[scale=0.8]

\coordinate (A) at  (0,0); 
\coordinate (B) at  (3,2.3);
\coordinate (C) at  (1,4);
\coordinate (D) at  (-3,1.5);
\coordinate (E) at  (0.65,2.2);
\coordinate (F) at  (-1,3.703);
\coordinate (G) at  (2.3,0.8);
\coordinate (H) at  (-2.9,1);
\coordinate (I) at  (-2.7,0.7);

\draw [ultra thick] (A) to [out=0,in=270] (B);
\draw [ultra thick] (B) to [out=90,in=0] (C);
\draw [ultra thick] (C) to [out=180,in=90] (D);
\draw [ultra thick] (D) to [out=270,in=180] (A);

 \node[above] at (-1.2,3.8) {$\Omega$};

\draw [ultra thick, blue] (F) to [out=295,in=200] (E);
\draw [ultra thick, blue] (E) to [out=20,in=130] (G);
\node at (0.65,2.25) {$\bullet$}; \node[below] at (0.8,2.2) {$x_M$};

\node[blue] at (-0.5,3) {$N^{\bar\bb}_\theta$};

\node[blue] at (-3.2,1) {$\bullet$}; \node[left, blue] at (H) {$x$};

\end{tikzpicture}
};

\node (B1) [below=0.6cm of A1] {
\begin{tikzpicture}[scale=0.8]

\coordinate (A) at  (0,0); 
\coordinate (B) at  (3,2.3);
\coordinate (C) at  (1,4);
\coordinate (D) at  (-3,1.5);
\coordinate (E) at  (0.74,2.37);
\coordinate (F) at  (-0.5,3.85);
\coordinate (G) at  (2.43,0.9);
\coordinate (H) at  (2.27,0.75);
\coordinate (I) at  (2.06,0.6);
\coordinate (J) at  (1.3,0.2);

\draw [ultra thick] (A) to [out=0,in=270] (B);
\draw [ultra thick] (B) to [out=90,in=0] (C);
\draw [ultra thick] (C) to [out=180,in=90] (D);
\draw [ultra thick] (D) to [out=270,in=180] (A);

 \node[above] at (-1.2,3.8) {$\Omega$};

\draw [ultra thick, blue] (F) to [out=265,in=220] (E);
\draw [ultra thick, blue] (E) to [out=40,in=165] (G);
\node at (0.75,2.65) {$\bullet$}; \node[below] at (0.9,2.37) {$x_M$};

\node[blue] at (0.3,3.5) {$\gamma_2^{\theta_n}$};

\draw [ultra thick, blue] (J) to [out=90,in=120] (I);
\node at (2.25,1.02) {$\bullet$}; \node[below] at (H) {$x$};

\node[blue] at (1.2,1.2) {$\gamma_1^{\theta_n}$};

\end{tikzpicture}
};

\node (B2) [below=0.6cm of A2] {
\begin{tikzpicture}[scale=0.8]

\coordinate (A) at  (0,0); 
\coordinate (B) at  (3,2.3);
\coordinate (C) at  (1,4);
\coordinate (D) at  (-3,1.5);
\coordinate (E) at  (0.74,2.37);
\coordinate (F) at  (-0.5,3.85);
\coordinate (G) at  (2.43,0.9);
\coordinate (H) at  (2.27,0.75);
\coordinate (I) at  (2.06,0.6);
\coordinate (J) at  (1.3,0.2);

\draw [ultra thick] (A) to [out=0,in=270] (B);
\draw [ultra thick] (B) to [out=90,in=0] (C);
\draw [ultra thick] (C) to [out=180,in=90] (D);
\draw [ultra thick] (D) to [out=270,in=180] (A);

 \node[above] at (-1.2,3.8) {$\Omega$};

\draw [ultra thick, blue] (F) to [out=295,in=200] (E);
\draw [ultra thick, blue] (E) to [out=20,in=130] (H);
\node at (0.88,2.68) {$\bullet$}; \node[below] at (0.74,2.37) {$x_M$};

\node[blue] at (0.4,3.7) {$N^{\bar\bb}_\theta$};

\node[blue] at (2.26,1.04) {$\bullet$}; \node[below, blue] at (H) {$x$};

\end{tikzpicture}
};

\draw[->, thick] (A1) -- (A2);
\draw[->, thick] (B1) -- (B2);

\end{tikzpicture}
\caption{\label{fig4}The two possible cases as $n\to+\infty$.}
\end{figure}

\noindent \emph{Case $1$: two $z_j^{\bb_n}$ collapse to $x\in\partial\Omega$.}\\
This is the case when, for all $n\in\N$, there exists a component $\gamma_1^{\theta_n}\subseteq N_{\theta_n}^{\bb_n}$ shrinking to $x$.
Passing to the limit as $n\to+\infty$, we have $\partial_{\bar\theta} u_{\bar\bb}(x)=0$ and $\partial_{\bar\theta} u_{\bar\bb}\not=0$ in a neighbourhood of $x$.
Without loss of generality, we can assume $x=\orig$, $\tangvet(\orig)=(1,0)$, and $\nu(\orig)=(0,-1)$. Thus $\partial_{\bar\theta} u_{\bar\bb}(\orig)=0$ and 
\[
\partial_{\bar\theta} u_{\bar\bb}>0,\quad\text{in }\overline\Omega\cap B_{\delta}(\orig)\setminus\{\orig\},
\]
(the negative case being completely analogous) for some $\delta>0$.
Consider the function $h(\theta)=\partial_\theta u_{\bar\bb}(\orig)=\cos\theta\partial_{x_1}u_{\bar\bb}(\orig)+\sin\theta\partial_{x_2}u_{\bar\bb}(\orig)$.
Then $h(\bar\theta)=0$ and $h'(\bar\theta)=\partial_{\theta^\perp} u_{\bar\bb}(\orig)\not=0$, since $\nabla u_{\bar\bb}\not=\orig$ on $\partial\Omega$ by Lemma~\ref{LemmaMP}. Assume $h'(\bar\theta)>0$ (this corresponds to the case $\bar\theta\in[0,\pi/2)\cup(3\pi/2,2\pi)$, the other case being analogous). In particular, $\partial_\theta u_{\bar\bb}(\orig)<0$ for all $\theta\in(\bar\theta-\eps,\bar\theta)$.
Since $\partial_{\bar\theta} u_{\bar\bb}(0,\delta/2)>0$, by continuity we obtain $\partial_\theta u_{\bar\bb}(0,\delta/2)>0$ for all $\theta\in(\bar\theta-\eps,\bar\theta)$, and hence, for each such $\theta$, there exists $t=t(\theta)\in(0,\delta/2)$ such that $\partial_\theta u_{\bar\bb}(0,t)=0$, that is
\begin{equation}
\label{eq:t}
(0,t)\in N^{\bar\bb}_\theta,
\end{equation}
since the conclusion of the lemma holds true at $\bar\bb$.

We claim that, for $\theta$ sufficiently close to $\bar\theta$, one has
\begin{equation}
\label{eq:tt}
N^{\bar\bb}_\theta\subseteq\{x\in\Omega:\mathrm{dist}(x,N^{\bar\bb}_{\bar\theta})<\delta/4\}.
\end{equation}
Assume the contrary. Then there exists $(\theta_n)_{n\in\N}\subseteq[0,2\pi)$ such that $\theta_n\to\bar\theta$ as $n\to+\infty$, and for all $n\in\N$, there exists $x_n\in N^{\bar\bb}_{\theta_n}$ with $\mathrm{dist}(x_n,N^{\bar\bb}_{\theta_n})>\delta/4$. Since $x_M\in N^{\bar\bb}_{\bar\theta}\cap N^{\bar\bb}_{\theta_n}$ for all $n$, we may also assume $\mathrm{dist}(x_n,N^{\bar\bb}_{\theta_n})<\delta/2$ and, possibly after reducing $\delta$, $\mathrm{dist}(x_M,\partial\Omega)>\delta$ and $|x_M-\orig|>2\delta$.
Passing to the limit as $n\to+\infty$, up to subsequences, we have $x_n\to\bar x\in\Omega$ and $\partial_{\bar\theta} u_{\bar\bb}(\bar x)=0$, which contradicts the fact that $\mathrm{dist}(\bar x,N^{\bar\bb}_{\theta_n})>\delta/4>0$, thus proving the claim.

Finally, observe that $\mathrm{dist}(\orig,N^{\bar\bb}_{\theta_n})>\delta$, which, together with~\ref{eq:t} and~\ref{eq:tt}, yields a contradiction. Therefore, Case~$1$ cannot occur.\\

\noindent\emph{Case $2$: three (or more) $z_j^{\bb_n}$ collapse to $x\in\partial\Omega$.}\\
In this case, for all $n\in\N$, there exist $\theta_n\in[0,2\pi)$ such that $\partial_{\theta_n}u_{\bb_n}$ changes sign at least three times on $\partial\Omega$ in a neighbourhood of $x$. Then, for all $n\in\N$, there exists $x_n\in\partial\Omega$, converging to $x$, such that $\partial_\tangvet \partial_{\theta_n}u_{\bb_n}(x_n)=\partial^2_{\tangvet\tangvet} \partial_{\theta_n}u_{\bb_n}(x_n)=0$. Passing to the limit as $n\to+\infty$, we deduce $\partial_\tangvet \partial_{\bar\theta} u_{\bar\bb}(x)=\partial^2_{\tangvet\tangvet} \partial_{\bar\theta} u_{\bar\bb}(x)=0$. Note that $M_\theta^{\bar\bb}=\emptyset$ implies $\partial_\nu \partial_{\bar\theta} u_{\bar\bb}(x)\not=0$.

As in the proof of Lemma~\ref{lemma:tg}, without loss of generality we can assume $x=\orig$ and parametrize $\partial\Omega$ in a neighbourhood of $\orig$ as in~\ref{par:bdd}, where such parametrization $\phi$ satisfies $\phi'(0)=0$ and $\phi''(0)>0$, since $\Omega$ has strictly positive boundary curvature. This yields $\partial_x \partial_{\bar\theta} u_{\bar\bb}(\orig)=\partial^2_{xx} \partial_{\bar\theta} u_{\bar\bb}(\orig)=0$ and $\partial_y \partial_{\bar\theta} u_{\bar\bb}(\orig)\not=0$. Hence, the Taylor expansion of $\partial_{\bar\theta} u_{\bar\bb}$ restricted to $\partial\Omega$ is given by (recalling that $\phi(\tau)=\frac12\phi''(0)\tau^2+\mathcal O(\tau^3)$)
\[
\partial_{\bar\theta} u_{\bar\bb}(\tau,\phi(\tau))=\frac12\partial_y \partial_{\bar\theta} u_{\bar\bb}(\orig)\phi''(0)\tau^2+\mathcal O(\tau^3).
\]
In particular, this implies that $\partial_{\bar\theta}u_{\bar\bb|\partial\Omega}$ does not change sign at $\orig$, contradicting the fact that the conclusion of the lemma holds at $\bar\bb$, so it should change sign there. 
\end{proof}

\subsection{End of the proof of Theorem~\ref{MainThm}}
\label{S2s2}

To conclude the proof of our main result, we need to show that the $\bb_1$ introduced in the previous subsection can be taken equal to $\bb_0$.

To this end, we define
\[
\pmb{\bb}=\inf\{\bb_1>\bb_0: M_\theta^\bb=\emptyset,\,\,\forall\theta\in[0,2\pi),\,\,\forall\bb>\bb_1\}.
\]
Observe that $\pmb{\bb}<+\infty$ by Lemma~\ref{lemma:MthetaBD2}.

Taking into account the results of Subsection~\ref{subsec1}, and in particular the claim of Lemma~\ref{thmBetaGrande}, the proof of Theorem~\ref{MainThm} reduces to showing that $\pmb{\bb}=\bb_0$.

\begin{prop}
\label{prop:beta}
We have
\[
\pmb{\bb}=\bb_0.
\]
\end{prop}
\begin{proof}
We argue by contradiction, assuming that $\pmb{\bb}>\bb_0$. By the definition of $\pmb{\bb}$, there exists $x_{\pmb{\bb}}\in M_\theta^{\pmb{\bb}}$ for some $\theta\in[0,2\pi)$. For the sake of readability, we simply write $u$ instead of $u_{\pmb{\bb}}$. 

We claim that $M_\theta^{\pmb{\bb}}\cap\Omega=\emptyset$. Otherwise, taking into account that $u_\theta$ satisfies
\[
-\Delta u_\theta-f'(u)u_\theta=0,\quad\text{in }\Omega,
\]
it is a famous result by Caffarelli and Friedman~\cite{CF85b} that $\partial_{\theta}u_{\bb_2}$ behaves as a homogeneous harmonic polynomial of degree $n\ge2$ in a neighbourhood of $x_{\pmb{\bb}}$. This means that, in a neighbourhood of $x_{\pmb{\bb}}$, the nodal set $N_\theta^{\pmb{\bb}}$ consists of $n$ smooth curves intersecting transversally at $x_{\pmb{\bb}}$ and forming equal angles. By stability, there is no nonempty domain $H\subseteq\Omega$, $|\Omega\setminus H|>0$, such that $\partial H\subseteq N_{\theta}^{\pmb{\bb}}$ (where the boundary of $H$ is considered as a subset of $\R^{2}$), see Remark~\ref{rmk:stab:Dir}. This implies that each of the $2n$ branches of the curves where $u_{\theta}=0$ and starting from $x_{\pmb{\bb}}$ must terminate at a distinct point on $\partial\Omega$, and at each such point $\partial_{\theta}u_{|\partial\Omega}$ changes sign, that is,~\ref{4pt} holds.
However, this contradicts the fact that $\partial_{\theta}u^\bb$ changes sign only twice for all $\bb>\pmb{\bb}$ and that, up to subsequences, $u^\bb \to u_{\pmb{\bb}}$ in $\mathcal C^{2,\alpha}(\Omega)$ as $\bb\to\pmb{\bb}$, by Lemma~\ref{prop} and the Ascoli-Arzelà Theorem.

Since $M_\theta^{\pmb{\bb}}\cap\Omega=\emptyset$, we have $x_{\pmb{\bb}}\in\partial\Omega$, and we now distinguish two cases.\\

\noindent\emph{Case 1: $\sharp\{\partial\Omega\cap N_\theta^{\pmb{\bb}}\}<3$.}\\
In this case, $\partial\Omega\cap N_\theta^{\pmb{\bb}}=\{x_{\pmb{\bb}},y\}$, and $\partial_\theta u_{|\partial\Omega}$ changes sign at both $x_{\pmb{\bb}}$ and $y$. The fact that  $\partial_\theta u_{|\partial\Omega}$ changes sign immediately follows since, being $\kappa>0$, there exist exactly two points $P,Q\in\partial\Omega$ such that $\nu(P)=\theta$, $\nu(Q)=-\theta$, and $\nu\neq\pm\theta$ on $\partial\Omega\setminus\{P,Q\}$, together with Lemma~\ref{LemmaMP}.   

By Lemma~\ref{lemma:tg}, the tangent vector to $\partial\Omega$ at $x_{\pmb{\bb}}$ is parallel to $\theta$. Without loss of generality, we may assume $x_{\pmb{\bb}}=\orig$, $\theta=\tangvet(\orig)=(1,0)$, and $\nu(\orig)=(0,-1)$. In particular, by convexity, we have $P=(P_1,P_2)$ with $P_1>0$ and $Q=(Q_1,Q_2)$ with $Q_1<0$, and
\begin{equation}
\label{der:normali}
u_{x_1}(P)<0,\quad u_{x_1}(Q)>0.
\end{equation}

Once again, we parametrize $\partial\Omega$ in a neighbourhood of $\orig$ as in~\ref{par:bdd}; the parametrization $\phi$ satisfies $\phi'(0)=0$ and $\phi''(0)>0$, since $\Omega$ has strictly positive boundary curvature. Since $\orig=x_{\pmb{\bb}}\in M_\theta^{\pmb{\bb}}$, it follows that $u_{x_1}(\orig)=u_{x_1x_1}(\orig)=u_{x_1x_2}(\orig)=0$.  

Differentiating twice the Robin boundary condition~\ref{eq:bdd:par} and evaluating at $\tau=0$, we obtain
\[
u_{x_1x_1x_2}(\orig)=-u_{x_2x_2}(\orig)\phi''(0)+\beta u(\orig)(\phi''(0))^2+\beta u_{x_2}(\orig)\phi''(0)>0,
\]
since $-u_{x_2x_2}(\orig)=f(u(\orig))\ge0$, $u(\orig)>0$, and $u_{x_2}(\orig)=-u_\nu(\orig)>0$ by Lemma~\ref{LemmaMP}.  

As $u_{x_1}$ satisfies $-\Delta u_{x_1}(\orig)-f'(u)u_{x_1}(\orig)=0$, we have
\[
u_{x_1x_1x_1}(\orig)=- u_{x_1x_2x_2}(\orig),
\]
and then the Taylor expansion of $u_{x_1}$ in a neighbourhood of $\orig$ is given by
\begin{equation}
\label{eqtaylor}
u_{x_1}(x_1,x_2)=ax_1^2+bx_1x_2-ax_2^2+o(x_1^2+x_2^2),
\end{equation}
where $a=\tfrac12 u_{x_1x_1x_1}(\orig)=-\tfrac12 u_{x_1x_2x_2}(\orig)$ and $b=u_{x_1x_1x_2}(\orig)>0$. Since $N_\theta^{\pmb{\bb}}$ consists of exactly one curve starting from $\orig$ and entering $\Omega$, we deduce $a=0$ (if $a\not=0$, there are two branches of curves entering $\Omega$ from the origin, which would imply that $\sharp\{\partial\Omega\cap N_\theta^{\pmb{\bb}}\}\ge3$). 
In particular, $N_\theta^{\pmb{\bb}}$ intersects $\partial\Omega$ orthogonally in $\orig$ and, since $u_{x_1x_1x_2}(\orig)>0$, we can deduce  that $u_{x_1}(\eta,\eta)>0$ and $u_{x_1}(-\eta,\eta)<0$ for all $\eta>0$ small enough. This implies $u_{x_1}(\tau,\phi(\tau))>0$ for $\tau\in(0,\delta)$ and $u_{x_1}(\tau,\phi(\tau))<0$ for $\tau\in(-\delta,0)$, $\delta>0$ small. However, this is in contradiction with~\ref{der:normali} and the fact that $\sharp\{\partial\Omega\cap N_\theta^{\pmb{\bb}}\}=2$.\\

\noindent\emph{Case $2$: $\sharp\{\partial\Omega\cap N_\theta^{\pmb{\bb}}\}\ge3$.}\\
We claim that this implies that~\ref{4pt} holds, that is
\[
N_{\tilde\theta}^{\pmb{\bb}}\cap\partial\Omega\supseteq\{z_1^{\pmb{\bb}},\dots,z_4^{\pmb{\bb}}\}\text{ and }\partial_{\tilde\theta}u_{|\partial\Omega}\text{ changes sign at }z_j^{\pmb{\bb}},\, j=1,\dots,4,
\]
for some $\tilde\theta\in[0,2\pi)$.  However, we have already shown that this leads to a contradiction, since $\partial_{\tilde\theta}u^\bb_{|\partial\Omega}$ changes sing exactly twice for all $\bb>\pmb{\bb}$.

Since $M_{\theta}^{\pmb{\bb}}\cap\Omega=\emptyset$, and $M_{\theta}^{\bb}=\emptyset$ for all $\bb>\pmb{\bb}$, for all $\theta$, Lemma~\ref{thmBetaGrande} implies that $u$ has a unique critical point $x_M\in\Omega$, which is a nondegenerate maximum.  

If~\ref{4pt} holds for $\tilde\theta=\theta$ we are done. Suppose this is not the case. Then, under the same normalization as in \emph{Case 1}, there exist two branches of $N_{\tilde\theta}^{\pmb{\bb}}$ emanating from $\orig$ (i.e.~$a\neq0$ in~\ref{eqtaylor}), which we denote by $\Gamma_1$ and $\Gamma_2$.  

Thus $x_M$ lies on one of the two arcs, say $x_M\in\Gamma_2$, and $\Gamma_1$ separates $\Omega$ into two distinct connected components. Then there exists a point $y$ in the component not containing $x_M$ such that $y\in N_{\theta'}^{\pmb{\bb}}$ for some $\theta'\neq\tilde\theta$. Since $M_{\tilde\theta}^{\pmb{\bb}}\cap\Omega=\emptyset$, $N_{\theta'}^{\pmb{\bb}}$ is a smooth curve, and it cannot cross $\Gamma_1$ because the only critical point is $x_M$. Moreover, by stability, $N_{\theta'}^{\pmb{\bb}}$ cannot enclose a subdomain of $\Omega$, so it must intersect $\partial\Omega$ in two distinct points where $\partial_{\theta'}u_{|\partial\Omega}$ changes sign.  

Clearly $x_M\in N_{\tilde\theta}^{\pmb{\bb}}$, and arguing as above we obtain two additional points in $N_{\tilde\theta}^\bb\cap\partial\Omega$ where $\partial_{\tilde\theta}u_{|\partial\Omega}$ changes sign, see Figure~\ref{fig-3}.
\end{proof}

\begin{figure}
\begin{tikzpicture}[scale=1.2]

\coordinate (A) at  (0,0); \node at (A) {$\bullet$}; \node[below] at (A) {$0$};
\coordinate (B) at  (3,2.3);
\coordinate (C) at  (1,4);
\coordinate (D) at  (-3,1.5);
\coordinate (E) at  (0.74,2.37);
\coordinate (F) at  (-0.5,3.85);

\draw [ultra thick] (A) to [out=0,in=270] (B);
\draw [ultra thick] (B) to [out=90,in=0] (C);
\draw [ultra thick] (C) to [out=180,in=90] (D);
\draw [ultra thick] (D) to [out=270,in=180] (A);

 \node[above] at (-1.2,3.8) {$\Omega$};

\draw [ultra thick, blue] (A) to [out=45,in=250] (C);	\node[blue] at (1,1) {$\Gamma_2$};
\draw [ultra thick, blue] (A) to [out=135,in=35] (D);	\node[blue] at (-2.2,2) {$\Gamma_1$};

\draw [ultra thick, red] (F) to [out=305,in=220] (E);
\draw [ultra thick, red] (E) to [out=40,in=220] (B);
\node at (E) {$\bullet$}; \node[below] at (1.1,2.4) {$x_M$};
\node[red] at (-0.2,2.5) {$N_{\theta'}^{\pmb{\bb}}$};

\draw [ultra thick, red] (-2.9,1) to [out=15,in=90] (-0.7,0.03) ;
\node[red] at (-2.05,1.35) {$N_{\theta'}^{\pmb{\bb}}$};
\node at (-1.45,0.95) {$\bullet$}; \node[below] at (-1.45,0.95) {$y$};

\end{tikzpicture}

\caption{\label{fig-3} In \emph{Case $2$} there is a direction $\theta'$ such that $\partial_{\theta'}u_{|\partial\Omega}$ changes sign at least $4$ times.}
\end{figure}

We can finally prove Theorem~\ref{MainThm}.

\begin{proof}[Proof of Theorem~\ref{MainThm}]
It is an immediate consequence of Lemma~\ref{thmBetaGrande} and Proposition~\ref{prop:beta}.
\end{proof}

\section{Particular cases: the proof of Corollary~\ref{MainThm2}}
\label{S3}

In this section we present several important cases in which Theorem~\ref{MainThm} can be applied. In particular, we treat cases (i)-(iii) of Corollary~\ref{MainThm2}, providing its proof. To this end, we show that problem~\hyperlink{PB}{$(\text{P}_\bb)$} admits a branch of stable solutions $u_\bb$ that is continuous in $\bb$ with respect to the $L^\infty(\Omega)$ norm.
To this end it is immediate to see that in the above problems we easily have the existence of a continuous branch of solutions required as assumption in Theorem~\ref{MainThm}.

Moreover, for $\bb$ sufficiently large, we also exhibit counterexamples to our main results as soon as the convexity assumption on the domain is dropped. In particular, we show that the conclusion may fail even for domains that are arbitrarily close to being convex, since in such cases one can find solutions with an arbitrarily large number of critical points.

\subsection{First eigenfunction}
\label{s2}
For all $\bb>0$, let $u_\bb$ be the first Robin eigenfunction, that is, the positive solution of
\[
\begin{cases}
-\Delta u_\bb=\lambda_\bb u_\bb&\text{in }\Omega\\
\partial_\nu u_\bb+\bb u_\bb=0&\text{on }\partial\Omega,
\end{cases}
\]
normalized so that $\norm{u_\bb}_{L^\infty(\Omega)}=1$. Here $\lambda_\bb:=\lambda_\bb(\Omega)$ denotes the first Robin eigenvalue in $\Omega$, which can be characterized as
\[
\lambda_\bb=\inf_{\substack{\psi\in H^1(\Omega)\\ \norm{\psi}_{L^2(\Omega)=1}}}\int_\Omega\abs{\nabla\psi}^2\,dx+\beta\int_{\partial\Omega}\psi^2\,d\sigma.
\]
 Since $u_\bb$ is stable, and taking into account that $\lambda_\bb\to\lambda_D:=\lambda_1(-\Delta,\Omega)$, i.e.~the first Dirichlet eigenvalue in $\Omega$, as $\bb\to+\infty$, it is immediate to see the existence of the continuous branch of solutions required in Theorem~\ref{MainThm}. Thus, we can conclude that $u_\bb$ has exactly one critical point for all $\bb>0$, which is in particular a nondegenerate maximum. Let us point out that, in this case, the nonlinearity $f$ depends on $\bb$. However, since $f(u_\bb)=\lambda_\bb u_\bb$ is uniformly bounded with respect to $\bb$, the proof of Theorem~\ref{MainThm} can be easily adapted to this setting.

\begin{rmk}
\begin{enumerate}
\item Clearly, uniqueness of the critical point does not hold for $\bb=0$, that is, in the case of Neumann boundary conditions. Indeed, in this situation the first eigenfunction is constant.
\item We recall that uniqueness of the critical point was already known for large values of $\bb$ as a consequence of~\cite[Theorem~1.1]{CF21} on strictly convex domains in any dimension.
\item We also recall that there exist convex domains for which $u_\bb$ is not $\log$-concave for small values of $\bb$, as shown by Andrews, Clutterbuck and Hauer in~\cite[Corollary~1.5]{ACH21}. Nevertheless, uniqueness of the critical point still holds.
\end{enumerate}
\end{rmk}

\subsection{Torsion problem}
\label{s1}

For $f\equiv1$, let $u_\bb$ be the Robin torsion function.

Note that, since the only harmonic function in $\Omega$ satisfying the Robin boundary condition is the constant function $0$, the solution $u_\bb$ is unique. Moreover, $u_\bb\not\equiv0$ (since $f(0)>0$) for all $\bb>0$ and its stability follows immediately from Definition~\ref{def:stab}, since $f'\equiv0$.
We know from~\cite[Theorem~1]{vdBB14} that
\[
\norm{u_\bb}_{L^\infty(\Omega)}\le \frac{C_1}{\lambda_\bb}\log\left(C_2\left(1+\frac{\sqrt\lambda_\bb}{\beta}\right)\right),
\]
for some dimensional constants $C_1>0$, $C_2>1$, where $\lambda_\bb$ denotes the first Robin eigenvalue in $\Omega$, consistently with the previous section. Taking into account that $\lambda_\bb\to\lambda_D$ as $\beta\to+\infty$, it follows that for all $\bb_0>0$ one has
\[
\norm{u_\bb}_{L^\infty(\Omega)}\le C,\quad\text{for all }\bb\ge\bb_0,
\]
for some $C=C(\bb_0)>0$. Thanks to the previous uniform bound, Lemma~\ref{prop} applies, and by the uniqueness of the solution it is then easy to see the continuity in $\bb$ of $u_\bb$. Therefore, we can apply Theorem~\ref{MainThm} to conclude that $u_\bb$ has exactly one critical point for all $\bb>0$.

We recall that this result is not new, since it was already proved in~\cite[Theorem~2]{Sak90}. Moreover, a stronger concavity property holds for large values of $\bb$ and in any dimension, provided that the domain is strictly convex. Indeed, in this case,~\cite[Theorem~1.2]{CF21} shows that $u_\bb$ is strictly $1/2$-concave.

\begin{rmk}
\label{rmkS1}
We conclude this part by pointing out that the result is no longer true if the domain is not convex, even if it is very close to being convex. In this case, uniqueness of the critical point may fail. More precisely, for all $\bb>0$, given any integer $k\in\N$ and any $\eps>0$, there exists a smooth and bounded domain $\Omega_\eps\subseteq\R^2$ such that, if $u_\bb^\eps$ solves
\[
\begin{cases}
-\Delta u_\bb^\eps=1&\text{in }\Omega_\eps\\
\partial_\nu u_\bb^\eps+\bb u_\bb^\eps=0&\text{on }\partial\Omega_\eps,
\end{cases}
\]
then
\begin{enumerate}[i)]
\item $\Omega_\eps$ is star-shaped with respect to an interior point,
\item $u_\bb^\eps$ has at least $k$ maximum points,
\item $\Omega_\eps$ locally converges to the strip $\mathcal S:=\R\times(0,1)$, i.e.~for all compact sets $\mathcal K\subseteq\R^2$ one has $\abs{\mathcal K\triangle (\Omega_\eps\cap\mathcal S)}\to0$, as $\eps\to0$,
\item the curvature of $\partial\Omega_\eps$ changes sign, vanishes at exactly two points and its minimum converges to $0$ as $\eps\to0$.
\end{enumerate}

This fact follows from the corresponding situation in the case of Dirichlet boundary conditions, as proved in~\cite{GG23}. Indeed, let $\Omega_\eps$ be the domain constructed in~\cite[Theorem~1.1]{GG23}, and let $u_D^\eps$ denote the Dirichlet torsion function on $\Omega_\eps$, which has $k$ nondegenerate maxima (the nondegeneracy can be deduced as in~\cite[Lemma~3.3]{DRG}). Then, by degree arguments and thanks to Lemma~\ref{ConvThm}, it is easy to see that there exists $\bar\bb>0$ such that
\[
(\bar\bb,+\infty)\subseteq\set{\bb>0:u_\bb^\eps\text{ has at least }k\text{ non degenerate maxima}},
\]
since $u_D^\eps$ has at least $k$ non degenerate maxima, by~\cite[Theorem 1.1]{GG23}.
\end{rmk}

\subsection{Minimal branch for nonlinear problems}
\label{s3}

It is not always possible to deal with a general nonlinearity $f$. Indeed, stable solutions may fail to exist.
We have already mentioned that this is the case for the Neumann problem ($\bb=0$) in convex domains, see~\cite{CH78,Mat79}. Furthermore, even for $\bb>0$, stable solutions may not exist in general. For instance,~\cite[Theorem~2.1]{BMMP} shows that if
\begin{gather*}
\int_{\partial\Omega}\bb^2 u_\bb^2\left(\bb-\kappa+\frac{f(u_\bb)}{\bb u_\bb}\right)\,d\sigma<0,\\
\beta+\min_{\partial\Omega}\kappa\ge0,
\end{gather*}
where $\kappa$ denotes the curvature of the boundary of $\Omega$, then the solutions of~\hyperlink{PB}{$(\text{P}_\bb)$} are unstable.

That being said, let us assume $f:=\lambda g$ where $\lambda>0$ and $g$ is a fixed smooth function satisfying
\begin{gather}
\label{g1}g(0)>0,\\
\label{g2}g'(t)>0,\quad\forall t\ge0.
\end{gather}
Typical examples are $g(t)=e^t$ or $g(t)=(1+t)^p$, $p>0$.
Fix any $\bb>0$. Under these assumptions, it is well known that there exists $\lambda_\bb^*:=\lambda_\bb^*(\Omega)>0$ such that, for all $\lambda\in(0,\lambda_\bb^*)$, the problem
\[
\hypertarget{PBl}{(\text{P}_\bb^\lambda)\qquad}
\begin{cases}
-\Delta u_\bb^\lambda=\lambda g(u_\bb^\lambda)&\text{in }\Omega\\
u_\bb^\lambda>0&\text{in }\Omega\\
\partial_\nu u_\bb^\lambda+\bb u_\bb^\lambda=0&\text{on }\partial\Omega,
\end{cases}
\]
admits a minimal solution $u_\bb^\lambda$. Here, minimal means that $u_\bb^\lambda\le U$ in $\Omega$ for any other solution $U$ of~\hyperlink{PBl}{$(\text{P}_\bb^\lambda)$}. It is also known that no solutions exist for $\lambda>\lambda_\bb^*$.
Moreover, $u_\bb^\lambda$ turns out to be stable. The same result holds for the corresponding Dirichlet problem. In this case, we denote by $\lambda_D^*:=\lambda_D^*(\Omega)>0$ the critical threshold and, for $\lambda\in(0,\lambda_D^*)$, we write $u_D^\lambda$ for the minimal stable solution. Note that $u_D^\lambda>0$ since $f(0)>0$.
We refer, for instance, to the book~\cite{BanB} for the classical theory about this problem.

Following~\cite{BanB}, let us recall that a function $u$ is called an \emph{upper solution} of problem~\hyperlink{PBl}{$(\text{P}_\bb^\lambda)$} if satisfies
\[
\begin{cases}
-\Delta u\ge\lambda g(u)&\text{in }\Omega\\
\partial_\nu u+\bb u\ge0&\text{on }\partial\Omega.
\end{cases}
\]

We can now state the main result of this section.

\begin{thm}
\label{ThmSTAB}
Let $\Omega\subseteq\R^2$ be a smooth and bounded domain such that the curvature $\kappa$ of its boundary is strictly positive, i.e.~$\kappa>0$ on $\partial\Omega$, and assume that $f=\lambda g$, where $g$ is a fixed smooth function satisfying~\ref{g1} and~\ref{g2}. Then, for all $\lambda\in(0,\lambda_D^*)$ and for all $\bb>0$, the minimal solution $u_\bb^\lambda$ of~\hyperlink{PBl}{$(\text{P}_\bb^\lambda)$} has a unique critical point. In particular, it is a nondegenerate maximum.
\end{thm}
\begin{proof}
It is enough to observe that for all $\bb_2\ge\bb_1>0$ one has
\[
\lambda_D^*\le\lambda_{\bb_2}^*\le\lambda_{\bb_1}^*.
\]
and that for all $\lambda\in(0,\lambda_D^*)$
\begin{equation}
\label{order}
0<u_D^\lambda\le u_{\bb_2}^\lambda\le u_{\bb_1}^\lambda,\quad\text{in }\Omega.
\end{equation}
Indeed, it is easy to see that, given $\lambda<\lambda_{\bb_1}^*$, the function $u_{\bb_1}^\lambda$ is an upper solution of~\hyperlink{PBl}{$(\text{P}_{\bb_2}^\lambda)$}. By the Picard iteration scheme, this yields the existence of $u_{\bb_2}^\lambda$ together with the monotonicity property $u_{\bb_2}^\lambda\le u_{\bb_1}^\lambda$. The same argument shows that $\lambda_D^*\le\lambda_{\bb_2}^*$ and that $u_D^\lambda\le u_{\bb_2}^\lambda$.

Thanks to the uniform bounds in~\ref{order}, Lemma~\ref{prop} applies, and it is then easy to show that the assumption on the existence of the continuous branch in Theorem~\ref{MainThm} is satisfied.
Therefore, the conclusion follows.
\end{proof}

\begin{rmk}
\label{rmkS2}
As in the case of the torsion problem, Theorem~\ref{ThmSTAB} no longer holds if the domain is not convex, even if it is arbitrarily close to being so. More precisely, if $g$ is increasing, convex, and satisfies $g(0)>0$, one can argue exactly as in Remark~\ref{rmkS1} to extend~\cite[Theorem~1.1]{DRG} to the case of Robin boundary conditions, for $\bb$ large and for $\lambda\in(0,\lambda_D^*(-1,1))$, where $\lambda_D^*(-1,1)$ denotes the threshold of the one-dimensional Dirichlet problem. This shows that there exist counterexamples that are arbitrarily close to convex domains for which uniqueness of the critical point fails.

To argue as in Remark~\ref{rmkS1}, one needs the nondegeneracy of the critical points of the corresponding Dirichlet solution $u_D^\eps$, which is not explicitly stated in~\cite[Theorem 1.1]{DRG}. Nevertheless, checking the proof of that theorem, it is straightforward to verify this property. Indeed, for any compact set $K\subseteq\Omega_\eps$, one has $\norm{u^\eps_D-(v_0+\eps\phi)}_{L^\infty(K)}\le C(K)\eps^2$, for some $C(K)>0$, where $v_0$ and $\phi$ are fixed functions such that $v_0+\eps\phi$ solves the linearized problem in a domain containing $\Omega_\eps$ and admits at least $k$ nondegenerate critical points, contained in a fixed, i.e.~independent from $\eps$, compact subsets of $\Omega_\eps$. Moreover, these critical points satisfy $\eps/c\le|\det\mathrm{Hess}{[v_0+\eps\phi]}|\le c\eps$, as $\eps\to0$, for some $c>0$,  see~\cite[Lemma 2.3]{DRG}. By classical interior regularity estimates, this implies $\norm{u^\eps_D-(v_0+\eps\phi)}_{\mathcal C^2(K)}\le C(K)\eps^2$, possibly with a larger constant $C(K)$.
This estimate, combined with the fact that the critical points of $v_0+\eps\phi$ are nondegenerate, contained in fixed compact subsets and  $|\det\mathrm{Hess}{[v_0+\eps\phi]}|$ has order exactly $\eps$, implies that $u^\eps_D$ admits at least $k$ nondegenerate critical points, as claimed.
\end{rmk}

\subsection*{Acknowledgments}

We would like to thank Zexi Wang for pointing out a mistake in a previous version of this manuscript.

\bibliographystyle{abbrv}
\bibliography{CriticalPointsRobin-CURV_0bis.bib}

\begin{thebibliography}{10}

\bibitem{app}
A.~Acker, L.~E. Payne, and G.~Philippin.
\newblock On the convexity of level lines of the fundamental mode in the
  clamped membrane problem, and the existence of convex solutions in a related
  free boundary problem.
\newblock {\em Z. Angew. Math. Phys.}, 32(6):683--694, 1981.

\bibitem{ACH21}
B.~Andrews, J.~Clutterbuck, and D.~Hauer.
\newblock Non-concavity of the {R}obin ground state.
\newblock {\em Camb. J. Math.}, 8(2):243--310, 2020.

\bibitem{BanB}
C.~Bandle.
\newblock {\em Isoperimetric inequalities and applications}, volume~7 of {\em
  Monographs and Studies in Mathematics}.
\newblock Pitman (Advanced Publishing Program), Boston, Mass.-London, 1980.

\bibitem{BMMP}
C.~Bandle, P.~Mastrolia, D.~D. Monticelli, and F.~Punzo.
\newblock On the stability of solutions of semilinear elliptic equations with
  {R}obin boundary conditions on {R}iemannian manifolds.
\newblock {\em SIAM J. Math. Anal.}, 48(1):122--151, 2016.

\bibitem{BDRG23}
L.~Battaglia, F.~De~Regibus, and M.~Grossi.
\newblock On the shape of solutions to elliptic equations in possibly non
  convex planar domains.
\newblock {\em Discrete Contin. Dyn. Syst.}, pages 1269--1318, 2023.

\bibitem{bl}
H.~J. Brascamp and E.~H. Lieb.
\newblock On extensions of the {B}runn-{M}inkowski and {P}r\'{e}kopa-{L}eindler
  theorems, including inequalities for log concave functions, and with an
  application to the diffusion equation.
\newblock {\em J. Functional Analysis}, 22(4):366--389, 1976.

\bibitem{cc}
X.~Cabr\'{e} and S.~Chanillo.
\newblock Stable solutions of semilinear elliptic problems in convex domains.
\newblock {\em Selecta Math. (N.S.)}, 4(1):1--10, 1998.

\bibitem{CF85b}
L.~A. Caffarelli and A.~Friedman.
\newblock Partial regularity of the zero-set of solutions of linear and
  superlinear elliptic equations.
\newblock {\em J. Differential Equations}, 60(3):420--433, 1985.

\bibitem{CH78}
R.~G. Casten and C.~J. Holland.
\newblock Instability results for reaction diffusion equations with {N}eumann
  boundary conditions.
\newblock {\em J. Differential Equations}, 27(2):266--273, 1978.

\bibitem{CF21}
G.~Crasta and I.~Fragal\`a.
\newblock Concavity properties of solutions to {R}obin problems.
\newblock {\em Camb. J. Math.}, 9(1):177--212, 2021.

\bibitem{DRG}
F.~De~Regibus and M.~Grossi.
\newblock On the number of critical points of stable solutions in bounded
  strip-like domains.
\newblock {\em J. Differential Equations}, 306:1--27, 2022.

\bibitem{DRG2}
F.~De~Regibus and M.~Grossi.
\newblock On the number of critical points of the second eigenfunction of the
  {L}aplacian in convex planar domains.
\newblock {\em J. Funct. Anal.}, 283(1):Paper No. 109496, 2022.

\bibitem{dgm}
F.~De~Regibus, M.~Grossi, and D.~Mukherjee.
\newblock Uniqueness of the critical point for semi-stable solutions in {$\Bbb
  R^2$}.
\newblock {\em Calc. Var. Partial Differential Equations}, 60(1):25, 2021.

\bibitem{emp}
P.~Esposito, M.~Musso, and A.~Pistoia.
\newblock On the existence and profile of nodal solutions for a two-dimensional
  elliptic problem with large exponent in nonlinearity.
\newblock {\em Proc. Lond. Math. Soc. (3)}, 94(2):497--519, 2007.

\bibitem{gnn}
B.~Gidas, W.~M. Ni, and L.~Nirenberg.
\newblock Symmetry and related properties via the maximum principle.
\newblock {\em Comm. Math. Phys.}, 68(3):209--243, 1979.

\bibitem{GG23}
F.~Gladiali and M.~Grossi.
\newblock On the number of critical points of solutions of semilinear equations
  in {$\Bbb R^2$}.
\newblock {\em Amer. J. Math.}, 144(5):1221--1240, 2022.

\bibitem{GG24}
F.~Gladiali and M.~Grossi.
\newblock On the critical points of solutions of {PDE} in non-convex settings:
  {T}he case of concentrating solutions.
\newblock {\em J. Funct. Anal.}, 287(11):Paper No. 110620, 2024.

\bibitem{GP23}
M.~Grossi and L.~Provenzano.
\newblock On the critical points of semi-stable solutions on convex domains of
  riemannian surfaces.
\newblock {\em Mathematische Annalen}, 2023.

\bibitem{hns}
F.~Hamel, N.~Nadirashvili, and Y.~Sire.
\newblock Convexity of level sets for elliptic problems in convex domains or
  convex rings: two counterexamples.
\newblock {\em Amer. J. Math.}, 138(2):499--527, 2016.

\bibitem{KawBook}
B.~Kawohl.
\newblock {\em Rearrangements and convexity of level sets in {PDE}}, volume
  1150 of {\em Lecture Notes in Mathematics}.
\newblock Springer-Verlag, Berlin, 1985.

\bibitem{ML71}
L.~G. Makar-Limanov.
\newblock The solution of the {D}irichlet problem for the equation {$\Delta
  u=-1$}\ in a convex region.
\newblock {\em Mat. Zametki}, 9:89--92, 1971.

\bibitem{Mat79}
H.~Matano.
\newblock Asymptotic behavior and stability of solutions of semilinear
  diffusion equations.
\newblock {\em Publ. Res. Inst. Math. Sci.}, 15(2):401--454, 1979.

\bibitem{Mbook}
J.~W. Milnor.
\newblock {\em Topology from the differentiable viewpoint}.
\newblock University Press of Virginia, Charlottesville, Va., 1965.

\bibitem{Sak90}
S.~Sakaguchi.
\newblock Uniqueness of the critical point of the solutions to some semilinear
  elliptic boundary value problems in {${\Bbb R}^2$}.
\newblock {\em Trans. Amer. Math. Soc.}, 319(1):179--190, 1990.

\bibitem{sperb}
R.~P. Sperb.
\newblock Extension of two theorems of {P}ayne to some nonlinear {D}irichlet
  problems.
\newblock {\em Z. Angew. Math. Phys.}, 26(6):721--726, 1975.

\bibitem{vdBB14}
M.~van~den Berg and D.~Bucur.
\newblock On the torsion function with {R}obin or {D}irichlet boundary
  conditions.
\newblock {\em J. Funct. Anal.}, 266(3):1647--1666, 2014.

\end{thebibliography}

\end{document}